\documentclass{article}
\usepackage[utf8]{inputenc}
\usepackage{bbm}
\usepackage{multirow}
\usepackage{lineno}

\usepackage[overload]{empheq}
\usepackage{mathrsfs}
\usepackage{amsfonts}
\usepackage{amsmath}
\usepackage{amssymb}
\usepackage{dsfont}
\usepackage{diagbox}
\usepackage{graphicx}
\usepackage{mathabx}

\usepackage{tikz}
\usepackage{caption}
\usepackage{subcaption}
\usepackage{placeins, microtype}
\usepackage{latexsym,amsthm,xcolor,graphicx, longtable, tabu}
\usepackage{hyperref}
\usepackage{natbib}
\usepackage{mathtools}
\mathtoolsset{showonlyrefs}

\makeatletter
\def\theorem@space@setup{
  \theorem@preskip=0.4cm 
  \theorem@postskip=\theorem@preskip 
}
\makeatother

\newtheorem{theorem}{Theorem}[section]

\newtheorem{lemma}[theorem]{Lemma}

\newtheorem{example}[theorem]{Example}

\newtheorem{conjecture}[theorem]{Conjecture}

\newcommand\N{\mathbb{N}}
\newcommand\R{\mathbb{R}}

\newcommand\Q{\mathbb{Q}}
\newcommand\PP{\mathbb{P}}
\newcommand\EE{\mathbb{E}}

\newcommand{\lp}[1]{\left(#1\right)}
\newcommand{\lb}[1]{\left[#1\right]}

\newcommand{\norm}[1]{\left|\left|#1\right|\right| }

\long\def\metanote#1#2{{\color{#1}\
\ifmmode\hbox\fi{\sffamily\mdseries\upshape [#2]}\ }}

\title{Rigidity of infinite exchangeable sequences with Gaussian marginals}
\author{Maximillian Newman \thanks{
    Section of Genetic Medicine, University of Chicago.
    Email: \texttt{mnewman98@uchicago.edu}
}}
\date{\today}

\begin{document}

\maketitle

\begin{abstract}
    We study infinite exchangeable sequences with Gaussian one-dimensional marginals. We formulate the conjecture that joint Gaussianity of a single pair of coordinates forces the entire sequence to be a Gaussian process. Although this conjecture remains open, we prove that joint Gaussianity of the first four coordinates is sufficient. We also establish the corresponding two-point criterion under the additional assumption that the directing measure is almost surely infinitely divisible.
\end{abstract}

\section{Introduction}

    For a sequence of observations $X = \lp{X_i}_{i=1}^{\infty}$, exchangeability is a natural relaxation of the notion of being independent and identically distributed (iid) as it preserves the marginal structure of the sequence. This relaxation allows one to still recover many of the properties of iid sequences by preserving some of the geometric and combinatorial structures of iid sequences, e.g. the classical works of Aldous \cite{aldous81}, Hoover \cite{hoover79}, and Kallenberg \cite{kallenberg89}, as well as recent works \cite{barber24, bladtschaiderman23, ramdasmanole26}. The most comprehensive survey for such rigidity results, to the best of our knowledge, remains \cite{kallenberg05}. These results are often strengthened versions of the classical work of de Finetti \cite{definetti37}.

    Consider an infinite sequence $X = \lp{X_i}_{i=1}^{\infty}$ of random variables on a probability space $\lp{\Omega, \mathcal{F}, \PP}$ invariant under the action of finite permutations of $\N$, i.e. for which
    \begin{equation*}
        \lp{X_1, \ldots, X_n} \stackrel{d}{=} \lp{X_{\sigma(1)}, \ldots, X_{\sigma(n)}} 
    \end{equation*}
    for any permutation $\sigma$ of $\{1,\ldots, n\}$ and any natural number $n$. De Finetti \cite{definetti37} established that there exists a unique random probability measure $\eta$ such that the sequence $X$ is independent and identically distributed conditional on $\eta$. 
    
    It is well known \cite[p. 30]{aldous85} that when $X$ is a Gaussian process if and only if $\eta$ admits a representation
    \begin{equation}\label{E: representation_result}
        \eta= \delta_{M} * \mu,
    \end{equation}
    where $M$ is a Gaussian random variable and $\mu$ is the law of a Gaussian random variable. If $X$ is not a Gaussian process, even with Gaussian marginals, this need not hold in general, as shown in the following example. 

    \begin{example}\label{X: rademacher}
        Let $\epsilon_i$ be an iid sequence of random variables, independent of a centered Gaussian random variable $Z$, for which
        \begin{equation*}
            \PP\lp{\epsilon_1 = 1} = \PP\lp{\epsilon_1 = -1} = \frac{1}{2}.
        \end{equation*}
        Then $X_i:= \epsilon_i Z$ yields an infinite exchangeable sequence of random variables for which the representation of Equation~\eqref{E: representation_result} fails to hold. Indeed, by construction we have 
        \begin{equation}\label{E: rademacher_example_measure}
            \eta = \frac{1}{2}\delta_{-Z} + \frac{1}{2}\delta_Z.
        \end{equation}
    \end{example}
    
    Note that exchangeability fixes the covariance structure of $X$. We elect to normalize $X$ so that $\rm{Var}(X_1) = 1$ and $\rm{Cov}(X_1,X_2) = \rho$ for some $\rho \in [0,1]$. For any $r$-tuple $(X_1,\ldots, X_r)$ the covariance matrix $\Sigma^{(r)} = \lp{\Sigma_{i,j}^{(r)}}_{i,j=1}^r$ is defined by
    \begin{equation*}
        \Sigma_{i,j}^{(r)} = 
        \begin{cases}
            1 &, \quad \text{ if } i = j\\
            \rho &, \quad \text{ if } i \neq j
        \end{cases}.
    \end{equation*}
    
    As the covariance structure is fixed by exchangeability and $X$ need not be a Gaussian process if each $X_i$ is marginally Gaussian, is there a counterexample if we make the stronger assumption that $(X_1,X_2)$ is jointly Gaussian? We indeed conjecture there is no such counterexample.
    \begin{conjecture}\label{C: conjecture}
        Let $X = \lp{X_i}_{i=1}^\infty$ be an infinite exchangeable sequence of random variables. If $(X_1,X_2)$ is a centered multivariate Gaussian, then $X$ is a Gaussian process.
    \end{conjecture}
    We make partial progress in this paper towards this 2-point conjecture.
    
    If $X$ were a finite sequence, this would surely not be the case, as can be seen by the following example.
    \begin{example}\label{E: finite perturbation}
        For any finite $K$, there exists a collection of exchangeable, centered, marginally Gaussian random variables $X_1,X_2,\ldots, X_K$ such that each pair $(X_i,X_j)$ (for distinct $i$ and $j$) is a multivariate Gaussian with covariance matrix $I_{2 \times 2}$ but for which no triple is Gaussian.

        Indeed, let $p$ denote the probability density function of a standard Gaussian random variable. Define a density $f$ on $\R^K$ by
        \begin{equation*}
            f(x) = \left(1+\epsilon\sum_{i=1}^{K-2}\sum_{j=i+1}^{K-1}\sum_{k=j+1}^K \sin(x_ix_jx_k)\right)\prod_{i=1}^K p(x_i)
        \end{equation*}
        for $\epsilon < \binom{K}{3}^{-1}$. Integrating out all but two of the marginals then yields a pairwise Gaussian kernel as $\sin$ is odd and $p$ is even. Yet no triple is Gaussian as integrating out all the other variables keeps the $\sin$ perturbation. $f$ then is the density of such a desired tuple.
    \end{example}
    
    While we are unable to prove Conjecture~\ref{C: conjecture}, we are able to prove a rigidity result under a $4$-point assumption. The significance of this result is that a finite amount of Gaussian information forces the entire de Finetti directing measure to have the form of Equation~\eqref{E: representation_result}. In this sense, the theorem rules out infinite exchangeable analogues of the finite-dimensional perturbations in Example~\ref{E: finite perturbation} once four coordinates are assumed to be jointly Gaussian. 
    \begin{theorem}\label{T: main_1}
        Let $X = \lp{X_i}_{i=1}^\infty$ be an infinite exchangeable sequence of random variables for which $(X_1, X_2, X_3, X_4)$ is a centered multivariate Gaussian with covariance structure $\Sigma^{(4)}$. Then $X$ is a Gaussian process.
    \end{theorem}
    Equivalently, the directing measure $\eta$ must almost surely admit the representation in Equation~\eqref{E: representation_result}. The proof is given in Section~\ref{S: proof of 1}.

    We are also able to prove the two-point conjecture under an additional structural assumption on the directing measure, namely almost sure infinite divisibility. The L\'{e}vy-Khintchine representation guaranteed by infinite divisibility gives a non-negativity property for the fourth cumulant, which provides enough rigidity to recover the Gaussian de Finetti representation of Equation~\eqref{E: representation_result}.
    \begin{theorem}\label{T: main_2}
        Let $X = \lp{X_i}_{i=1}^\infty$ be an infinite exchangeable sequence of random variables with directing measure $\eta$ that is almost surely infinitely divisible. If $(X_1,X_2)$ is a centered multivariate Gaussian with covariance structure $\Sigma^{(2)}$, then $X$ is a Gaussian process. 
    \end{theorem}
    Theorem~\ref{T: main_2} is proven in Section~\ref{S: proof of 2}. The pairwise Gaussian assumption in Theorem~\ref{T: main_2} cannot be replaced by one-dimensional Gaussianity alone, even when the directing measure is almost surely infinitely divisible. The following example shows that randomizing the conditional variance of an otherwise Gaussian directing measure preserves the Gaussian marginal but destroys joint Gaussianity.
    
    \begin{example}
        Let $V$ be a non-constant random variable taking values in $(0,1)$. Conditional on $V$, we let $M$, conditional on $V$, be a centered normal random variable with variance $1-V$. We define the almost surely infinitely divisible directing measure
        \begin{equation*}
            \eta = \delta_M *\gamma_{V}.
        \end{equation*}
        
        Let $X_1, X_2, \ldots$ be independent and identically distributed with law $\eta$. Then each $X_i$ is marginally a standard Gaussian, and yet $(X_1,X_2)$ is not jointly Gaussian unless $V$ is deterministic. Indeed, if it were we would have that $X_1-X_2$ is Gaussian. In particular, we would have that its fourth moment satisfies
        \begin{equation*}
            \EE\lb{(X_1-X_2)^4} = 12\EE\lb{V^2} = 3(2 \EE\lb{V})^2,
        \end{equation*}
        which fails because $V$ is non-constant.
    \end{example}

    The remaining difficulty in Conjecture~\ref{C: conjecture} is to upgrade information about averaged two-point marginals into an almost sure structural statement about the directing measure. The two-point Gaussian assumption imposes identities only after averaging over the random measure $\eta$ that constrain $\eta$, but it is not obvious to the author how these restrictions imply the conjecture. The proof of Theorem~\ref{T: main_1} overcomes this by using four-point Gaussianity to control the second-order fluctuations of the relevant averaged identities, forcing an almost sure multiplicative relation for the moment generating function of the directing measure. The proof of Theorem~\ref{T: main_2} overcomes the same obstruction by leveraging the additional structure given by the L\'{e}vy-Khintchine representation of infinitely divisible measures. In the unrestricted two-point problem, neither of these additional mechanisms is available, and this is the source of the remaining gap.

\section{Proof of Theorem~\ref{T: main_1}}\label{S: proof of 1}

    We fix $X = \lp{X_i}_{i=1}^\infty$ an exchangeable sequence of real-valued random variables. By de Finetti's Theorem \cite{definetti37} there exists a random probability measure $\eta$ such that, conditional on $\eta$, the sequence $X$ is independent and identically distributed. Let $\lambda$ denote the law of $\eta$ as a random measure on $\R$.

    We now introduce some notation. Fix $\rho \in [0,1]$. For $t\in \R$, we define the moment generating function $M_\eta$ and the log moment generating function $K_\eta$ by
    \begin{equation*}
        M_\eta(t) := \int_\R e^{tx} d\eta(x) \quad \text{ and } \quad K_\eta(t):= \log M_\eta(t),
    \end{equation*}
    respectively, whenever $M_\eta(t) < \infty$. We define a tilt $g_t(\eta)$ of the moment generating function by
    \begin{equation}\label{E: g_defn}
        g_t(\eta) := \exp\lp{K_\eta(t)- \frac{1+\rho}{2}t^2}.
    \end{equation}
    The law of a centered Gaussian random variable with variance $\rho$ we denote by $\gamma_\rho$, and we set $\gamma = \gamma_1$.
    We define the defect $q_{a,b}$, for any $a,b \in \R$, by 
    \begin{equation}\label{E: q_defn}
        q_{a,b}:=g_ag_b - e^{2\rho ab}g_{a+b}.
    \end{equation}
    
    Observe that if $\eta = \delta_N * \gamma_{1-\rho}$ for $N \sim \gamma_\rho$, then the defect $q_{a,b}$ will be identically zero for any $a,b \in \R$ almost surely. It is this relationship to the representation in Equation~\eqref{E: representation_result} that our method exploits. 

    The proof proceeds by expressing Gaussianity of finite-dimensional marginals as a rigidity statement for the tilted law of the directing measure. Lemma~\ref{L: exponential_K_moment} first shows that exponential moments of finite tuples of $X$ are encoded by the log moment generating function $K_\eta$ of the directing measure. Lemma~\ref{L: g_moment} translates the Gaussian assumption on $(X_1,X_2, X_3, X_4)$ into explicit mixed moment identities for the tilted quantities $g_t$. The key step is Lemma~\ref{L: multiplicative_deficit_is_zero}; using the four-point Gaussian assumption, we compute the $L^2(\lambda)$ norm of the multiplicative defect $q_{a,b}$ from Equation~\eqref{E: q_defn} and show that this norm is zero. Thus the averaged tilt identities become an almost sure multiplicative identity. Lemma~\ref{L: eta_decomposition} then converts this identity into additivity of $K_\eta$, first on the rationals $\Q$ and then on all of $\R$ by continuity. This then forces the desired representation of Equation~\eqref{E: representation_result}.

    The following lemma demonstrates that the log moment generating function characterizes the full joint distribution of any finite tuple from $X$.

    \begin{lemma}\label{L: exponential_K_moment}
        For $n \geq 1$ suppose that $a_1,\ldots, a_n \in \R$ satisfies $M_\eta(a_j) < \infty$ almost surely for all $j=1,\ldots, n$. Then
        \begin{equation*}
            \EE\lb{\exp\lp{\sum_{i=1}^n K_\eta(a_i)}} = \EE\lb{\exp\lp{\sum_{i=1}^n a_i X_i}}.
        \end{equation*}
    \end{lemma}
    \begin{proof}
        Conditional on $\eta$, the variables $X_1,\ldots, X_n$ are independent and identically distributed with common law $\eta$. Therefore
        \begin{equation*}
            \EE\lb{\exp\lp{\sum_{i=1}^n a_i X_i} \,\mid\, \eta} = \prod_{i=1}^n \int_\R e^{a_i x} d\eta(x) = \exp\lp{\sum_{i=1}^n K_\eta(a_i)}.
        \end{equation*}
        Taking expectations then gives the result.
    \end{proof}

    We now verify that $K_\eta$ and $M_\eta$ are almost surely finite, and that $K_\eta$ is continuous.
    \begin{lemma}\label{L: mgf_finite_continuous}
        For any $t \in \R$ we have that $M_\eta(t)$ and $K_\eta(t)$ are almost surely finite. Moreover, $K_\eta$ is almost surely continuous.
    \end{lemma}
    
    \begin{proof}
        Since $X_1$ is a standard Gaussian random variable, for every fixed $t\in \R$ we have
        \begin{equation*}
            \EE\lb{e^{tX_1}}=e^{t^2/2}<\infty.
        \end{equation*}
        By de Finetti's theorem,
        \begin{equation*}
            \EE\lb{M_\eta(t)}
            =
            \EE\lb{\int_\R e^{tx}\,d\eta(x)}
            =
            \EE\lb{\EE\lb{e^{tX_1}\mid \eta}}
            =
            \EE\lb{e^{tX_1}}
            <
            \infty.
        \end{equation*}
        Since $M_\eta(t)$ is nonnegative, this implies that $M_\eta(t)<\infty$ almost surely for each fixed $t\in \R$.
    
        We now upgrade this fixed-$t$ statement to a simultaneous statement for all $t\in \R$. Since $\Q$ is countable, there is an event $E$ of probability one such that
        \begin{equation*}
            M_\eta(q)<\infty
            \qquad
            \text{for every } q\in \Q
        \end{equation*}
        on $E$. Fix $\omega\in E$ and let $t\in \R$. Choose rational numbers $q_-<t<q_+$. Write
        \begin{equation*}
            t=\theta q_-+(1-\theta)q_+
        \end{equation*}
        for some $\theta\in(0,1)$. By H\"older's inequality,
        \begin{align*}
            M_\eta(t)
            &=
            \int_\R e^{tx}\,d\eta(x) \\
            &=
            \int_\R \left(e^{q_-x}\right)^\theta
                    \left(e^{q_+x}\right)^{1-\theta}
                \,d\eta(x) \\
            &\leq
            \left(\int_\R e^{q_-x}\,d\eta(x)\right)^\theta
            \left(\int_\R e^{q_+x}\,d\eta(x)\right)^{1-\theta} \\
            &=
            M_\eta(q_-)^\theta M_\eta(q_+)^{1-\theta}
            <
            \infty.
        \end{align*}
        Therefore, on $E$, $M_\eta(t)<\infty$ for every $t\in \R$.
    
        Since $e^{tx}>0$ and $\eta$ is a probability measure, we also have
        \begin{equation*}
            M_\eta(t)>0
            \qquad
            \text{for every } t\in \R.
        \end{equation*}
        Hence $K_\eta(t)=\log M_\eta(t)$ is finite for every $t\in \R$ on the same event $E$.
    
        Finally, the same H\"older inequality shows that $K_\eta$ is convex. Indeed, for $s,t\in \R$ and $\theta\in[0,1]$,
        \begin{equation*}
            K_\eta\left(\theta s+(1-\theta)t\right)
            \leq
            \theta K_\eta(s)+(1-\theta)K_\eta(t).
        \end{equation*}
        Thus, on $E$, $K_\eta$ is a finite convex function on all of $\R$. Every finite convex function on $\R$ is continuous. Therefore $K_\eta$ is almost surely continuous.
    \end{proof}

    \begin{lemma}\label{L: g_moment}
        Assume that $X_1$ is Gaussian distributed with variance $1$ and that for some $n \in \{1,2,3,4\}$ that the vector $(X_1,\ldots, X_n)$ is a centered multivariate Gaussian with covariance matrix $\Sigma^{(n)}$. Then for every $a_1,\ldots, a_n \in \R$ we have
        \begin{equation*}
            \EE\lb{\prod_{i=1}^n g_{a_i}} = \exp\lp{-\frac{\rho}{2} \sum_{i=1}^n a_i^2 + \rho \sum_{1 \leq i < j \leq n} a_ia_j}.
        \end{equation*}
    \end{lemma}
    \begin{proof}
        By definition of $g$ in Equation~\eqref{E: g_defn},
        \begin{equation}\label{E: tilt_moment_1}
            \prod_{i=1}^n g_{a_i} = \exp\lp{\sum_{i=1}^n K_\eta(a_i) - \frac{1+\rho}{2} \sum_{i=1}^n a_i^2}.
        \end{equation}
        By Lemma~\ref{L: mgf_finite_continuous} we can apply Lemma~\ref{L: exponential_K_moment}.
        By Lemma~\ref{L: exponential_K_moment}, taking the expectation of Equation~\eqref{E: tilt_moment_1} yields
        \begin{equation}\label{E: tilt_moment_2}
            \EE\lb{\prod_{i=1}^n g_{a_i}} = \exp\lp{-\frac{1+\rho}{2} \sum_{i=1}^n a_i^2} \EE\lb{\exp\lp{\sum_{i=1}^n a_i X_i}}.
        \end{equation}
        As $(X_1,\ldots, X_n)$ is a centered Gaussian with covariance matrix $\Sigma$ we have that
        \begin{equation}\label{E: tilt_moment_3}
            \EE\lb{\exp\lp{\sum_{i=1}^n a_i X_i}} = \exp\lp{\frac{1}{2}\sum_{i,j = 1}^n a_ia_j \Sigma_{ij}} = \exp\lp{\frac{1}{2}\sum_{i=1}^n a_i^2 + \rho \sum_{1 \leq i < j \leq n} a_ia_j}.
        \end{equation}
        Substituting Equation~\eqref{E: tilt_moment_3} into Equation~\eqref{E: tilt_moment_2} then yields the claim.
    \end{proof}

    \begin{lemma}\label{L: multiplicative_deficit_is_zero}
        Assume that $(X_1,X_2, X_3,X_4)$ is a centered multivariate Gaussian with covariance matrix $\Sigma^{(4)}$.
        Then for every $a,b \in \R$, 
        \begin{equation*}
            q_{a,b} = 0 \quad \text{ almost surely.}
        \end{equation*}
    \end{lemma}

    \begin{proof}
        We proceed by computing the $L^2(\lambda)$ norm of $q_{a,b}$. By linearity of the inner product we have
        \begin{equation}\label{E: q_norm_1}
            \norm{q_{a,b}}_{L^2(\lambda)}^2 = \EE\lb{g_a^2 g_b^2} - 2e^{2\rho ab} \EE\lb{g_ag_bg_{a+b}} + e^{4\rho ab} \EE\lb{g_{a+b}^2}.
        \end{equation}
        We now apply Lemma~\ref{L: g_moment} three times: For $n=4$ and $(a_1, a_2, a_3,a_4) = (a,a,b,b)$ we obtain
        \begin{equation}\label{E: q_norm_2}
            \EE\lb{g_a^2 g_b^2} = e^{4\rho ab}.
        \end{equation}
        For $n = 3$ and $(a_1,a_2,a_3) = (a,b,a+b)$ we obtain
        \begin{equation}\label{E: q_norm_3}
            \EE\lb{g_a g_b g_{a+b}} = e^{2\rho ab}.
        \end{equation}
        With $n = 2$ and $(a_1,a_2) = (a+b,a+b)$ we obtain
        \begin{equation}\label{E: q_norm_4}
            \EE\lb{g_{a+b}^2} = 1.
        \end{equation}
        Combining Equations~\eqref{E: q_norm_1}, \eqref{E: q_norm_2}, and \eqref{E: q_norm_3} with Equation~\eqref{E: q_norm_4} then yields
        \begin{equation*}
            \norm{q_{a,b}}_{L^2(\lambda)}^2 = e^{4\rho ab} - 2 e^{4 \rho ab} + e^{4 \rho ab} = 0.
        \end{equation*}
        Therefore $q_{a,b} = 0$ almost surely.
    \end{proof}

    \begin{lemma}\label{L: eta_decomposition}
        Assume for any fixed $a,b \in \R$ that $q_{a,b} = 0$ almost surely.
        Then there exists a real-valued random variable $N = N(\eta)$ such that with probability one
        \begin{equation}\label{E: K_quadratic}
            K_\eta(t) = Nt + \frac{1-\rho}{2}t^2 \quad \text{ for all } t \in \R.
        \end{equation}
        Moreover,
        \begin{equation*}
            \eta = \delta_N * \gamma_{1-\rho}
        \end{equation*}
        almost surely.
    \end{lemma}

    \begin{proof}
        We first claim that there is an event $E$ of probability one such that for every $a,b \in \Q$ that
        \begin{equation*}
            K_{\eta}(a+b) = K_\eta(a) + K_\eta(b) + (1-\rho) ab.
        \end{equation*}
        Indeed, if $q_{a,b} = 0$, we have that
        \begin{equation*}
            g_a g_b = e^{2\rho ab} g_{a+b} \quad \text{ a.s.}
        \end{equation*}
        Substituting the definition of $g$ from Equation~\eqref{E: g_defn} we obtain
        \begin{equation*}
            \exp\lp{K_\eta(a) - \frac{1+\rho}{2}a^2} \exp\lp{K_\eta(b) - \frac{1+\rho}{2} b^2} = e^{2\rho ab} \exp\lp{K_\eta(a+b) - \frac{1+\rho}{2}(a+b)^2}.
        \end{equation*}
        Taking logarithms and rearranging then gives
        \begin{equation*}
            K_\eta(a+b) = K_{\eta}(a) + K_\eta(b) + (1-\rho)ab.
        \end{equation*}
        As this holds almost surely for any fixed $a,b$ we obtain that with probability one that it holds for all $a,b \in \Q$ by countability of $\Q$. 

        We proceed now to prove Equation~\eqref{E: K_quadratic}. Define the linear part $L_\eta$ of $K_\eta$ by
        \begin{equation*}
            L_\eta(t) := K_\eta(t) - \frac{1-\rho}{2}t^2.
        \end{equation*}
        Then for every $a,b \in \Q$ we have
        \begin{equation*}
            L_\eta(a+b) = L_\eta(a) + L_\eta(b).
        \end{equation*}
        Therefore $L_\eta$ is additive on $\Q$. Note that $K_\eta$, as the log moment generating function of a probability measure, is continuous, implying continuity of $L_\eta$. By continuity it follows therefore that
        \begin{equation*}
            L_\eta(t) = tL_\eta(1) \quad \text{ for all } t \in \R.
        \end{equation*}
        Defining $N \equiv N(\eta) := L_\eta(1)$ we obtain Equation~\eqref{E: K_quadratic}. By definition of $K_\eta$ we further obtain
        \begin{equation*}
            \int_\R e^{tx} d\eta(x) = \exp\lp{Nt + \frac{1-\rho}{2}t^2},
        \end{equation*}
        which, conditional on $N$, is the moment generating function of a Gaussian random variable with mean $N$ and variance $1-\rho$. This completes the claim.
    \end{proof}

    We are now able to prove Theorem~\ref{T: main_1}
    \begin{proof}
        We proceed by showing that $\eta$ satisfies the representation from Equation~\eqref{E: representation_result}. By Lemma~\ref{L: multiplicative_deficit_is_zero} we obtain that $q_{a,b} = 0$ almost surely for every fixed $a,b \in \R$. In particular, this holds almost surely jointly for all $a,b \in \Q$. By Lemma~\ref{L: eta_decomposition} there exists a random variable $N = N(\eta)$ such that
        \begin{equation}\label{E: conv}
            \eta = \delta_N * \gamma_{1-\rho}.
        \end{equation}
        It suffices, therefore, to show that $N$ is a centered Gaussian random variable with variance $\rho$. Indeed, as $X_1$ is normally distributed with variance $1$ Equation~\eqref{E: conv} implies that
        \begin{equation*}
            \EE\lb{e^{tX_1}} = e^{\frac{1}{2}t^2} = \EE\lb{e^{Nt + \frac{1-\rho}{2}t^2}} = \EE\lb{e^{Nt}} e^{\frac{1}{2}(1-\rho)t^2}.
        \end{equation*}
        Therefore
        \begin{equation*}
            \EE\lb{e^{Nt}} = e^{\frac{1}{2} \rho t^2},
        \end{equation*}
        which is the moment generating function of a centered Gaussian random variable with variance $\rho$. Therefore $X$ is a Gaussian process.
    \end{proof}

\section{Proof of Theorem~\ref{T: main_2}}\label{S: proof of 2}

    In this section, we prove Theorem~\ref{T: main_2}. The proof exploits the extra rigidity imposed by the L\'{e}vy-Khintchine representation theorem. Without infinite divisibility, the two-point Gaussian assumption only constrains certain averaged conditional moments of $\eta$, leaving room for cancellations between fluctuations of the conditional variance and higher order conditional cumulants. Lemma~\ref{L: cumulant_calculation} rules out this cancellation mechanism; the fourth cumulant of $\eta$ is non-negative and vanishes only when $\eta$ is Gaussian. We apply this to $X_1 - X_2$. Conditional on $\eta$, Lemma~\ref{L: fourth_moment_to_cumulants} expresses the fourth moment of $X_1 - X_2$ in terms of the conditional variance of $\eta$ and the fourth cumulant of $\eta$. Unconditionally, however, $X_1 - X_2$ is Gaussian by the two-point assumption. Comparing these two fourth-moment calculations then forces the conditional variance and cumulants to be constant, which yields the result after a short calculation.

    \begin{lemma}\label{L: cumulant_calculation}
        Let $\lambda$ be an infinitely divisible probability measure on $\R$ with finite fourth moment. Let $\kappa_4(\lambda)$ denote its fourth cumulant. Then $\kappa_4(\lambda) \geq 0$. Moreover, $\kappa_4(\lambda) = 0$ if and only if $\lambda$ is a (possibly degenerate) Gaussian law.
    \end{lemma}
    \begin{proof}
        Since $\lambda$ is infinitely divisible, its characteristic function $\hat{\lambda}$ has L\'{e}vy-Khintchine representation
        \begin{equation*}
            \hat{\lambda}(t) = \exp\lp{iat - \frac{b}{2}t^2 + \int_\R \lp{e^{itx} - 1 - itx\mathbf{1}_{\{|x| \leq 1\}}} d\nu(x)},
        \end{equation*}
        where $a \in \R$, $b \geq 0$, and $\nu$ is a measure on $\R$ satisfying
        \begin{equation*}
            \nu(\{0\}) = 0 \quad \text{ and } \quad \int_\R \lp{1 \wedge x^2} d\nu(x) < \infty.
        \end{equation*}
        Since $\lambda$ has a finite fourth moment we have
        \begin{equation*}
            \int_\R x^4 d\nu(x) < \infty.
        \end{equation*}
        Moreover, differentiating the logarithm of the characteristic function four times at the origin yields
        \begin{equation*}
            \kappa_4(\lambda) = \int_\R x^4 d\nu(x).
        \end{equation*}
        Hence $\kappa_4(\lambda) \geq 0$.

        If $\kappa_4(\lambda) = 0$, then $\nu$ is the zero measure. In particular, $\hat{\lambda}$ simplifies to
        \begin{equation*}
            \hat{\lambda}(t) = \exp\lp{iat - \frac{b}{2}t^2},
        \end{equation*}
        which is the characteristic function of a Gaussian random variable with mean $a$ and variance $b$.
    \end{proof}

    \begin{lemma}\label{L: fourth_moment_to_cumulants}
        Let $\lambda$ be a probability measure on $\R$ with finite fourth moment. Let $Y_1$ and $Y_2$ be independent with common law $\lambda$. Let $v(\lambda)$ and $\kappa_4(\lambda)$ be the variance and fourth cumulant of $\lambda$, respectively. Then
        \begin{equation*}
            \EE\lb{(Y_1-Y_2)^4} = 12v(\lambda)^2 + 2 \kappa_4(\lambda).
        \end{equation*}
    \end{lemma}
    \begin{proof}
        Cumulants add under convolution, and the fourth cumulant is unchanged by multiplication by $-1$. Therefore the second and fourth cumulants of $Y_1-Y_2$, denoted by $\kappa_2(Y_1-Y_2)$ and $\kappa_4(Y_1-Y_2)$, are $2v(\lambda)$ and $2\kappa_4(\lambda)$, respectively. Since $Y_1-Y_2$ is centered, its fourth moment is related to its cumulants by
        \begin{equation*}
            \EE\lb{(Y_1-Y_2)^4} = \kappa_4(Y_1-Y_2) + 3 \kappa_2(Y_1-Y_2)^2.
        \end{equation*}
        The result then follows by substitution.
    \end{proof}

    We are now prepared to prove Theorem~\ref{T: main_2}.
    \begin{proof}
        We proceed by demonstrating that we may write the directing measure $\eta$ as
        \begin{equation*}
            \delta_M*\gamma_{1-\rho},
        \end{equation*}
        for $M$ a Gaussian random variable with variance $\rho$.

        We first observe that, because $X_1$ is a standard normal, we have that
        \begin{equation*}
            \EE\lb{X_1^4} = 3.
        \end{equation*}
        By de Finetti's Theorem, we have
        \begin{equation*}
            \EE\lb{\int_\R x^4 d\eta(x)} = \EE\lb{X_1^4} = 3 < \infty,
        \end{equation*}
        so $\eta$ almost surely has finite fourth moment.

        Define $M$, $V$, and $K$ to be the conditional mean, variance, and fourth cumulant, i.e.
        \begin{equation*}
            M := \int_\R x d\eta(x), \quad V := \int_\R (x-M)^2 d\eta(x), \quad \text{ and } \quad K := \kappa_4(\eta).
        \end{equation*}
        By Lemma~\ref{L: cumulant_calculation} we have that $K \geq 0$ almost surely. As $X_1$ and $X_2$ are conditionally independent with common law $\eta$, Lemma~\ref{L: fourth_moment_to_cumulants} yields
        \begin{equation*}
            \EE\lb{(X_1-X_2)^4 \mid \eta} = 12 V^2 + 2K.
        \end{equation*}
        Taking expectations yields
        \begin{equation}\label{E: 4th_moment}
            \EE\lb{(X_1-X_2)^4} = 12\EE\lb{V^2} + 2 \EE\lb{K}.
        \end{equation}
        As $(X_1,X_2)$ is a centered multivariate Gaussian with covariance structure $\Sigma^{(2)}$, we have that $X_1 - X_2$ is a centered Gaussian with variance $2(1-\rho)$. In particular, Equation~\eqref{E: 4th_moment} then yields
        \begin{equation*}
            12(1-\rho)^2 = 12 \EE\lb{V^2} + 2 \EE\lb{K} \geq 12(1-\rho)^2 + 2 \EE\lb{K}
        \end{equation*}
        by Jensen's inequality.
        We therefore have that $K \equiv 0$ almost surely and $V$ almost surely equal to $1-\rho$.

        As $K$ is almost surely equal to zero and $\eta$ is almost surely infinitely divisible, we have that $\eta$ is almost surely a Gaussian law. In particular, $\eta$ is a Gaussian law with mean $M$ and variance $1-\rho$ conditional on $M$.

        We now check that $M$ is indeed Gaussian distributed. To this end, we observe, using the L\'{e}vy-Khintchine representation, that
        \begin{equation*}
            \EE\lb{e^{itX_1} \mid \eta} = \exp\lp{itM - \frac{1}{2}(1-\rho)t^2}.
        \end{equation*}
        Taking expectations and using the fact that $X_1$ is marginally a standard normal random variable we have that
        \begin{equation*}
            \EE\lb{e^{itM}} = e^{-\frac{\rho}{2}t^2},
        \end{equation*}
        proving the claim. Therefore $X$ is a Gaussian process.
    \end{proof}

\bibliographystyle{alpha}
\bibliography{main} 

\end{document}